\documentclass[lm,a4paper,twoside,fleqn,leqno,gray,md]{pre} 

\usepackage[applemac]{inputenc}
\usepackage{math}
\usepackage{microtype}

\theoremstyle{plain}
\newtheorem{lem}{Lemma}[section]
\newtheorem*{main}{\kam Theorem}
\newtheorem*{step}{Step Lemma}
\newtheorem*{iter}{Iterative Lemma}

\def\Pb{\smash{P^\circ}}
\def\Pt{\smash{\tilde P}}
\def\Ph{\smash{\hat P}}
\def\pt{\tilde p}

\def\ob{\bar\om}
\def\i {\mathrm{i}}
\def\sgt{\tilde\sg}
\let\th\theta

\begin{document}

\title  {\kam à la R}
\author {Jürgen Pöschel}
\date   {July 2010, Version 2.1}

\address{Institut für Analysis, Dynamik und Optimierung\\
         Universität Stuttgart, Pfaffenwaldring 57, D-70569 Stuttgart\\
         poschel@mathematik.uni-stuttgart.de}

\maketitle
\infotrue


In~\cite{R-09} Rüssmann proposed -- quoting from his abstract --
“a new variant of the \kam-theory, containing an artificial parameter $q$, $0<q<1$, which makes the steps of the \kam-iteration infinitely small in the limit $q\nnearrow1$. \dots\ 
The new technique of estimation differs completely from all what has appeared about \kam-theory in the literature up to date.
Only Kolmogorov's idea of local linearization and
Moser's modifying terms are left.
The basic idea is to use the polynomial structure in order to transfer, at least partially, the whole \kam-procedure outside of the
original domain of definition of the given dynamical system.”

It is the purpose of this note to make this scheme accessible in an even simpler setting, namely for analytic perturbations of constant vector fields on a torus. As a side effect the result may be the shortest complete \kam proof for perturbations of integrable vector fields available so far.

\section{Result}

Let $N$ denote a constant vector field on the \m{n}-torus $\Tn=\Rn/2\pi\Zn$ describing uniform rotational motions with frequencies $\om = (\om_1,\dots,\om_n)$. Putting $N$ into normal form, we have $N=\om$. A small perturbation
$
  X = N+P
$
usually destroys this simple flow, due to frequency drifts and the effect of resonances. If, however, the frequencies $\om$ are strongly nonresonant, the perturbation $P$ is sufficiently smooth and small, and if we are allowed to add a small correctional \m{n}-vector to adjust frequencies, then $X$ is conjugate to $\om$. This is the content of the classical \kam theorem with modifying terms for this model problem, as introduced by Moser~\cite{Mo67}.

The precise setting is the following. We consider $N$ as a vector field in normal form depending on the frequencies $\om$ as parameters. These vary in some neighbourhood of a fixed compact set $\Om\subset\Rn$ consisting of strongly nonresonant frequencies. That is, each $\om\in\Om$ satisfies
\[
  \eqlabel{dio}
  \n{\ipr{k,\om}} \ge \frac{\al}{\Dl(\n{k})}, 
  \qq 0\ne k\in\Zn,
\]
with some $\al>0$ and some \emph{Rüssmann approximation function $\Dl$}. These are continuous, increasing, unbounded functions $\Dl\maps[1,\iny)\to[1,\iny)$ such that $\Dl(1)=1$ and
\[
  \int_1^\iny \frac{\log\Dl(t)}{t^2}\dt < \iny.
\]

The perturbation $P$ is assumed to be analytic in the angles $\th\in\Tn$ and may depend analytically on the parameters~$\om$ as well. The complex domains are
\[
  D_s = \set{\th: \n{\Im\th}<s},
  \qq
  \Om_h = \set{z: \n{z-\Om}<h},
\]
where $\n{\cdd}$ denotes the \emph{max}-norm for complex vectors, while it denotes the \emph{sum}-norm for integer vectors. To simplify matters considerably, we employ the \emph{weighted norms}
\[
  \eqlabel{norm}
  \n{P}_{s,h} = \sup_{\om\in\Om_h} \sum_{k\in\Zn} \n{p_k(\om)}\e^{\n{k}s},
  \qq
  P = \sum_{k\in\Zn} p_k(\om)\e^{\i\ipr{k,\th}}.
\]
Finally, with any approximation function $\Dl$ we associate another such function $\Lm$ by setting $\Lm(t)=t\Dl(t)$.

\begin{main}
Suppose $X=N+P$ is real analytic on $D_s\x\Om_h$ with
\[
  \eqlabel{small}
  \n{P}_{s,h} = \ep < \frac{h}{16} \le \frac{\al}{32\Lm(\ta)},
\]
where $\ta$ is so large that
\[
  r \defeq 8\int_{\ta}^\iny \frac{\log\Lm(t)}{t^2} \dt < \frac{s}{2}.
\]
Then there exists a real map $\ph\maps\Om\to\Om_h$, and for each $\om\in\Om$ a real analytic diffeomorphism $\Phi_\om$ of the \m{n}-torus, such that
\[
  \Phi_\om^*(\ph(\om)+P) = \om.
\]
Moreover, $\n{\ph-\id}_\Om \le \ep$ and $\n{\Phi-\id}_{s-2r,\Om} \le \Lm(\ta)\al\inv\ep$.
\end{main}

To keep things as simple as possible, we do not aim to optimize our constants, nor do we address regularity questions with respect to~$\om$. 

The above smallness condition does not depend explicitly on the dimension~$n$ of the problem. However, this dimension enters implicitly through the small divisor conditions~\eqref{dio} and Dirichlet's lemma which states that for nonresonant vectors $\om$,
\[
  \min_{0<\n{k}\le K} \n{\ipr{k,\om}} \le \frac{\n{\om}}{K^{n-1}}.
\]
Hence the approximation function $\Dl$ has to grow at a rate depending on~$n$ in order to obtain admissible frequencies. A typical example is $\Dl(t) = t^\nu$ with $\nu>n-1$.

\section{Outline}

We prove the theorem by an iterative process of successive coordinate transformations proposed by Kolmogorov~\cite{Kol}. However, at variance with the crustimoney proseedcake~\cite[Chapter IV]{WP}, we use a scheme of estimates proposed by Rüssmann, which does not rely on superlinear convergence speeds, but aims to decrease the size of the perturbation just a tiny bit at each step.

To this end, we split $P$ into an ‘infrared’ part $\Pt$ and an ‘ultraviolet’ part~$\Ph$.
However -- and this is a new twist -- $\Ph$ also contains fractions of the Fourier coefficients of \emph{low order}. As a result, $\Pt$ will be bounded on a \emph{larger domain}, with even a \emph{better bound} than $P$ itself.

The term $\Pt$ is then handled as usual. We write the coordinate transformation $\Phi$ as the time-\m1-map of the flow $F_t$ of a vector field~$F$, which solves the homological equation
$
  \lie{F,N} = \Pt-\Pb
$,
where $\Pb$ denotes the mean value of $\Pt$ and $\lie{\cd,\cd}$ the Lie bracket of two vector fields.
We obtain
\begin{align*}
  \Phi^*(N+\Pt) 
  &= \rbar{F_t^*(N+\Pt)}_{t=1} \\
  &= N+\lie{N,F} + \int_0^1(1-t)F_t^*\lie{\lie{N,F},F}\dt \\
  &\qq {} + \Pt + \int_0^1 F_t^*\lie{\Pt,F}\dt.
\end{align*}
Using the homological equation, the result is
\[
  \Phi^*(N+\Pt) 
  = N + \Pb + \int_0^1 F_t^*\lie{(t\Pt+(1-t)\Pb,F}\dt.
\]
The new constant vector field is then
\[
  N_\pl = N+\Pb = \om+p_0(\om) = \om_\pl.
\] 
A change of parameters $\ph\maps \om_\pl\mapsto\om$ transforms this back into standard normal form.
Taking into account the discarded ultraviolet term~$\Ph$, 
\[
  P_\pl = F_1^*\Ph + \int_0^1 F_t^*\lie{P_t,F}\dt,
  \qq
  P_t = t\Pt+(1-t)\Pb,
\]
is then the new perturbation of the new vector field $N_\pl$.

In carrying out the pertinent estimates, it is extremely convenient to use weighted norms such as~\eqref{norm} instead of sup-norms as it is done in~\cite{R-09}. This way, estimates on larger domains such as~\eqref{pt} and small divisor estimates such as~\eqref{Fs} are immediate and do not require results from harmonic analysis. Indeed, it could be argued that sup-norms should be converted to weighted norms before entering the \kam machinery -- once inside they are rather clumsy, inefficient, and produce a lot of unwieldy constants.

\section{Step lemma}

Before proceeding to the details we note that by a proper scaling of time and hence of the vector fields we can assume that the small divisor conditions~\eqref{dio} hold with the normalized value
$\al=2$. 

In the following everything will be \emph{real analytic} without explicitly saying so.

\begin{step}
Let $0<\sg<s/2$ and $\ta\ge1$. Set $a=1-\e^{-\ta\sg}$ and assume
\[
  \eqlabel{hyp}
  \n{P}_{s,h} \le \ep < \min \set{\dfrac{h}{2a}, \dfrac{1}{2\Lm(\ta)}},
  \qq
  h \le \frac{1}{\Lm(\ta)}.
\]
Then there exist parameter and coordinate transformations
$\ph\maps\Om_{h-2a\ep}\to\Om_h$ and $\Phi\maps D_{s-2\sg}\to D_s$
which together transform $X=N+P$ into $X_\pl = N+P_\pl$ with
\[
  \n{P_\pl}_{s-2\sg,h-2\ep} \le q\ep,
\]
where
\[
  q = (1-a+a^2b)(1+b)\e^a,
  \qq
  b = \Lm(\ta)\ep.
\]
Moreover, 
$
  \n{\ph-\id}_{h-2a\ep} \le a\ep
$ and $
  \n{\Phi-\id}_{s-2\sg,h-2a\ep} \le \Lm(\ta)\sg\ep
$.
\end{step}

\begin{proof}
Let $P = \Pt+\Ph$ with
\[
  \Ph
  = \sum_{\n{k}\ge\ta} p_k \e^{\i\ipr{k,\th}}
    + (1-a)\sum_{\n{k}<\ta} p_k\e^{\n{k}\sg}\e^{\i\ipr{k,\th}}.
\]
Clearly, in view of $\e^{-\ta\sg}=1-a$,
\[
  \n{\Ph}_{s-\sg} \le (1-a)\n{P}_s \le (1-a)\ep,
\]
where we dropped the ›$h$‹ from the notation since it stays fixed until the very last paragraph of this section.
On the other hand, the polynomial rest
\[
  \Pt = \sum_{0<\n{k}<\ta} \pt_k\e^{\i\ipr{k,\th}},
  \qq
  \pt_k = (1-(1-a)\e^{\n{k}\sg})p_k,
\]
is bounded on a \emph{larger} domain. Indeed, with $\sgt = \sg(1-a)/a$, 
\[
  \eqlabel{pt}
  \n{\Pt}_{s+\sgt}
   \le \sup_{0\le t\le\ta}\pas*{1-(1-a)\e^{t\sg}}\e^{t\sgt} \! 
       \sum_{\n{k}<\ta} \n{p_k}\e^{\n{k}s}
   \le a\ep,
\]
as the function under the sup is monotonically decreasing for $0\le t\le\ta$ and equals $a$ at~$t=0$.

The linearized equation $\lie{F,N} = \Pt-\Pb$ is solved as usual by
\[
  F = \sum_{0<\n{k}<\ta} \frac{\pt_k}{\i\!\ipr{k,\om}}\e^{\i\ipr{k,\th}}.
\]
For any $\om\in\Om_h$ there is $\ob\in\Om$ with $\n{\om-\ob}<h\le1/\Lm(\ta)$ by~\eqref{hyp} and hence, in view of $\Lm(\ta)=\ta\Dl(\ta)$,
\[
  \n{\ipr{k,\om-\ob}} 
  \le \n{k}\n{\om-\ob}
  \le \ta h
  \le \frac{\ta}{\Lm(\ta)}
  = \frac{1}{\Dl(\ta)}.
\]
As $\ob$ satisfies~\eqref{dio} with $\al=2$, all relevant divisors thus admit the lower bound
\[
  \n{\ipr{k,\om}}
  \ge \n{\ipr{k,\ob}}-\n{\ipr{k,\om-\ob}}
  \ge \frac{2}{\Dl(\ta)} - \frac{1}{\Dl(\ta)}
  =   \frac{1}{\Dl(\ta)}.
\]
So with~\eqref{pt} and $a=1-\e^{-\ta\sg}\le \ta\sg$ we get
\[
  \eqlabel{Fs}
  \n{F}_{s+\sgt}
  \le \Dl(\ta) \n{\Pt}_{s+\sgt} 
  \le \Dl(\ta)a\ep
  \le \Lm(\ta)\sg\ep.
\]
In particular, $\n{F}_{s+\sgt} \le \sg$ by hypothesis~\eqref{hyp}, so the vector field~$F$ generates a flow $F_t$ that for $0\le t\le1$ satisfies
\[
  F_t\maps D_{s-2\sg}\to D_{s-\sg},
  \qq
  \n{F_t-\id}_{s-2\sg} \le \Lm(\ta)\sg\ep.
\]

To estimate $P_\pl$ we note that 
$\n{P_t}_{s+\sgt} \le t\n{\Pt}_{s+\sgt}+(1-t)\n{\Pb}_{s+\sgt} \le a\ep$ and
\[
  (s+\sgt)-(s-\sg) = \frac{1-a}{a}\sg+\sg = \frac{\sg}{a}.
\]
Lemma~\ref{lie} of Appendix~\ref{a-wn} and the abbreviation $b = \Lm(\ta)\ep$ thus yield
\[
  \n1{\lie{P_t,F}}_{s-\sg}
  \le \frac{a}{\sg} \n{P_t}_{s+\sgt}\n{F}_{s+\sgt} 
  \le \Lm(\ta)a^2\ep^2
  = a^2b\ep.
\]
In view of~\eqref{Fs} we can apply Lemma~\ref{trans} of Appendix~\ref{a-wn} with $r=s-\sg$ and $\lm=1/a$ to obtain
\[
  \int_0^1 \n1{F_t^*\lie{P_t,F}}_{s-2\sg}\dt
  {}\le  (1+b)\e^a\n1{\lie{P_t,F}}_{s-\sg} 
  {}\le  a^2b(1+b)\e^a\ep.
\]
Similarly,
\[
  \n1{F_1^*\Ph}_{s-2\sg}
  \le (1+b)\e^a \n{\Ph}_{s-\sg} 
  \le (1-a)(1+b)\e^a\ep.
\]
Both estimates together yield the stated estimate of $P_\pl$.

Finally, $N_\pl = N+\Pb$ has frequencies $\om_\pl = \om+p_0(\om)$.
As 
\[
  \n{p_0}_h \le \n{\Pt}_{s+\sgt,h} \le a\ep < h/2
\]
by hypothesis~\eqref{hyp}, the map $\om\mapsto\om_\pl$ has an inverse 
\[
  \ph\maps\Om_{h-2a\ep}\to\Om_{h-a\ep}, \q \om = \ph(\om_\pl),
\] 
satisfying $\n{\ph-\id}_{h-2a\ep} \le a\ep$ by Lemma~\ref{inv} of Appendix~\ref{a-inv}. Then $N_\pl$ is again in standard normal form, and the proof of the Step Lemma is complete.
\end{proof}

\section{Iteration and convergence}

Iterating the Step Lemma is simple. We can always choose $0<a<1$ and $0<b\le1/2$ so that
\[
  q = (1-a+a^2b)(1+b)\e^a < 1,
\]
and we can even make $q$ as close to $1$ as we wish.
It then suffices to choose for $\ep$, $h$ and $\Lm$ geometric sequences with the same base~$q$, namely
\[
  \ep_\nu = \ep_0q^\nu, \qq
  h_\nu = h_0q^\nu, \qq
  \Lm_\nu = \Lm_0q^{-\nu},
\]
where we assume that $\Lm_0 \ge \Lm(1)=\Dl(1)$.
Next, let $\ta_\nu = \sup\set{\ta: \Lm(\ta)\le\Lm_\nu}$ and define $\sg_\nu$ and $s_\nu$ through
\[
  1-a = \e^{-\ta_\nu\sg_\nu},
  \qq
  s_{\nu+1}=s_\nu-2\sg_\nu.
\]
As we will see in a moment, the $s_\nu$ have a positive limit for $\Lm_0$ sufficiently large.

\begin{iter}
Suppose that 
\[
  \n{P}_{s_0,h_0} \le \ep_0 < \min \set{\dfrac{1-q}{2a} h_0 , \dfrac{b}{\Lm_0}},
  \qq
  h_0 \le \frac{1}{\Lm_0},
\]
with $\Lm_0$ sufficiently large. Then for each $\nu\ge1$ there exists a parameter and coordinate transformation 
\[
  (\Phi_\nu,\ph_\nu)\maps D_{s_\nu}\x\Om_{h_\nu} \to D_{s_0}\x\Om_{h_0}
\]
which transforms $N+P_0$ into $N+P_\nu$ such that $\n{P_\nu}_{s_\nu,h_\nu} \le \ep_\nu$.
\end{iter}

\begin{proof}
This follows by applying the Step Lemma repeatedly and composing the resulting mappings. Just note that $\Lm(\ta_\nu) \le \Lm_\nu$ and
\[
  \ep_\nu \Lm_\nu = \ep_0\Lm_0 \le b \le 1/2,
  \qq
  \frac{h_\nu-2a\ep_\nu}{h_{\nu+1}} = \frac{h_0-2a\ep_0}{qh_0} \ge 1,
\]
for all $\nu$ by construction and hypotheses.
\end{proof}

\begin{proof}[Proof of the KAM Theorem]
Recall that we normalized $\al=2$. Then the hypotheses of the Iterative Lemma are satisfied by $P_0=P$ with $\ep_0=\ep$, $h_0=h$, $s_0=s$ and the same $\Lm_0$ as above.

To be able to apply this lemma infinitely often we have to verify that the $s_\nu$ tend to a positive limit. Indeed, 
\[
  \smash[b]{\sum_{\nu\ge1}} \frac{1}{\ta_\nu}
  \le \int_0^\iny \frac{\upd\nu}{\Lm\inv(\Lm_0 q^{-\nu})}
  =   \frac{1}{\log q\inv} \int_{\ta_0}^\iny \frac{\upd\Lm(t)}{t\Lm(t)}
\]
via letting $t = \Lm\inv(\Lm_0 q^{-\nu})$.
Integrating by parts and requiring $\Lm(\ta_0)\ge q\inv$ we get
\[
  \sum_{\nu\ge0} \frac{1}{\ta_\nu}
  \le \frac{1}{\log q\inv} \int_{\ta_0}^\iny \frac{\log\Lm(t)}{t^2}\dt. 
\]
It follows that
\[
  r
  \defeq \sum_{\nu\ge0} \sg_\nu
  = \sum_{\nu\ge0} \frac{\log(1-a)\inv}{\ta_\nu}
  \le \frac{\log(1-a)}{\log q} \int_{\ta_0}^\iny \frac{\log\Lm(t)}{t^2}\dt. 
\]
Hence, by choosing $\ta_0$ sufficiently large, we can achieve that $r<s/2$
and thus $s_\nu \ssearrow s-2r > 0$.

For the statement of the theorem we choose $a=1/2$ and $b=1/16$, which results in
\[
  q \approx \frac{9}{10}, \qq \frac{\log(1-a)}{\log q} \le 8.
\]
In fact, $7$ instead of $8$ would also do, but  we prefer powers of~$2$.

Now fix $\om\in\Om$, and consider the sequence of transformations provided by the Iterative Lemma for this~$\om$. On $D_r$ the family $(\Phi_\nu)$ is uniformely bounded, as is the vector sequence $(\ph_\nu)$. So by Montel's and Weierstrass' theorem, there is a convergent subsequence. As 
$
  \n{P_\nu}_{s_\nu,h_\nu} \le \ep_\nu \to 0
$
along any such subsequence, this subsequence transforms $X=N+P$ at the limit parameter value $\tilde\om$ into the normal form~$N$ at $\om$, that is, the constant vector field~$\om$.
\end{proof}

\appendix

\section{An inverse function theorem}  \label{a-inv}

\begin{lem} \label{inv}
Suppose $f\maps\Om_{h}\to\Cn$ is analytic and 
\[
  \n{f-\id}_h \le \ep < h/2.
\] 
Then $f$ has an analytic inverse $\ph\maps\Om_{h-2\ep}\to\Om_h$, and $\n{\ph-\id}_{h-2\ep}\le\ep$.
\end{lem}

\begin{proof}
For any $0<k<h-2\ep$ we have
\[
  \n{Df-I}_{k+\ep} \le \frac{\ep}{h-(k+\ep)} < 1
\]
by Cauchy's inequality. Therefore, the operator
\[
  T\maps \ph \mapsto \id -(f-\id)\comp\ph
\]
defines a contraction on the space of analytic maps $\ph\maps\Om_{k}\to\Om_h$, $\n{\ph-\id}_{k}\le\ep$. Its unique fixed point $\ph$ is the analytic inverse to $f$ on $\Om_k$. Letting $k\to h-2\ep$ we obtain the claim.
\end{proof}

\section{Weighted norms}  \label{a-wn}

\begin{lem} \label{lie}
Let $U$ and $V$ be analytic vector fields on the torus~$\Tn$. Then,
for $0<r<\min\set{u,v}$,
\[
  \n1{\lie{U,V}}_{r}
  \le 
  \frac{1}{\e} \pas{\frac{1}{u-r}+\frac{1}{v-r}} \n{U}_{u}\n{V}_{v}.
\]
\end{lem}

\begin{proof}
We have
$U = \sum_{k} u_ke_k$ and $V = \sum_k v_ke_k$
with $e_k = \e^{\i\ipr{k,\th}}$. Therefore,
\[
  DU\cd V
   = \sum_{k} \i u_k\ipr{k,V}e_k 
  {}= \sum_{k,l} \i u_k\ipr{k,v_l}e_{k+l}  
  {}= \sum_{k,l} \i u_k\ipr{k,v_{l-k}}e_l
\]
and thus
\begin{align*}
  \n{DU\cd V}_{r}
  &\le \sum_{k,l} \n{k}\n{u_k}\n{v_{l-k}}\e^{\n{l}r} \\
  &\le \sum_{k,l} \n{k}\n{u_k}\e^{\n{k}r}\n{v_{l-k}}\e^{\n{l-k}r} 
   \displaybreak[0] \\
  &\le \sup_{t\ge0} t\e^{-(u-r)t}
       \smash{\pas3{\sum_{k} \n{u_k}\e^{\n{k}u}}}
       \smash{\pas3{\sum_{l} \n{v_{l}}\e^{\n{l}v}}}
       \\
  &\le \frac{1}{\e(u-r)}\n{U}_u\n{V}_v.
\end{align*}
Exchanging the roles of $U$ and $V$ we get an analogous estimate for $\n{DV\!\cd U}_r$ which proves the claim.
\end{proof}

\begin{lem} \label{trans}
Suppose the vector fields $F$ and $V$ are analytic on the torus~$\Tn$. 
If $b = \sg\inv \n{F}_{r+\lm\sg} \le 1/2$ with $0<\sg<r$ and $\lm>0$, then
\[
  \n{F_t^*V}_{r-\sg} \le (1+bt)\e^{1/\lm}\n{V}_r , 
  \qq
  0\le t\le 1.
\]
\end{lem}

\begin{proof}
We have the Lie series expansion
\[
  F_t^*V = \sum_{n\ge0} \frac{1}{n\fac}V_nt^n
\]
with $V_0=V$ and $V_n = \lie{V_{n-1},F}$ for $n\ge1$.
Let $\n{\cd}_i = \n{\cd}_{r-i\sg/n}$ for $0\le i\le n$.
Then, by the preceding lemma,
\begin{align*}
  \n{V_n}_{r-\sg}
  &=\n{V_n}_n
   = \n{\lie{V_{n-1},F}}_n \\
  &\le \pas{\frac{n}{\e\sg}+\frac{1}{\e\lm\sg}} \n{V_{n-1}}_{n-1} \n{F}_{r+\lm\sg} \\
  &= \frac{n}{\e\sg} \pas{1+\frac{1}{\lm n}} \n{V_{n-1}}_{n-1} \n{F}_{r+\lm\sg}.
\end{align*}
Applying this step $n$ times, we get
\[
  \n{V_n}_{r-\sg} \le \pas{\frac{n}{\e\sg}}^n \e^{1/\lm} \n{V}_r\n{F}_{r+\lm\sg}^n.
\]
Summing up and replacing $\sg\inv \n{F}_{r+\lm a}$ by $b$ we obtain
\[
  \n{F_t^*V}_{r-\sg}
  \le \n{V}_r \e^{1/\lm} \sum_{n\ge0} \frac{1}{n\fac} \pas{\frac{nb t}{\e}}^n .
\]
With $n\fac \ge n^n/\e^{n-1}$ for $n\ge1$ and $0\le bt\le 1/2$  the last sum is bounded by
\[
  1 + \sum_{n\ge1} \frac{(bt)^n}{\e} \le 1+bt.
  \qedhere
\]
\end{proof}

If $V$ depends on parameters $\om$ in a point set $\Pi$, we define the norm 
$\n{V}_{s,\Pi}$ as in~\eqref{norm}.
All the preceding estimates are uniform with respect to such parameters, so the results extend to this case.

\medskip

\textit{Acknowledgement.}
It is a pleasure to thank the referees for pointing out a number of imprecisions, in particular a slight misquotation of Winnie-the-Pooh and a mistaken application of my non-existent French.

\vskip2pc
\footnotesize
\textit{MSC 2010:} 37J40, secondary 70H08, 70K43.
\par
\textit{Keywords:} \kam theory, invariant tori, small divisors, weighted norms.
\vskip2pc

\putaddress

\end{document}